\documentclass[12pt]{article}
\usepackage[cp1251]{inputenc}
\usepackage[english]{babel}
\usepackage{amsmath,amsthm,amssymb}
\usepackage{hyperref}
\newtheorem{theorem}{\noindent Theorem}
\newtheorem{lemma}{\noindent Lemma}
\newtheorem{corollary}{\noindent Corollary}
\newtheorem{defe}{\noindent Definition}
\newtheorem{Th}{Theorem}
\newtheorem{Lm}[Th]{Lemma}
\newtheorem{Co}[Th]{Corollary}
\newtheorem{proposition}[Th]{Proposition}

\def\eps{\varepsilon}

\DeclareMathOperator{\thi}{th}

\date{}
\author{A.~M.~Vershik, P.~B.~Zatitskiy, F.~V.~Petrov
}

\title{VIRTUAL CONTINUITY OF MEASURABLE FUNCTIONS OF SEVERAL VARIABLES AND EMBEDDINGS THEOREMS}

\begin{document}

\maketitle

\abstract{Classical Luzin's theorem states that the measurable function
of one variable is ``almost'' continuous. This is not so anymore for functions of several variables.
The search of right analogue of the Luzin theorem leads to a notion of virtually continuous
functions of several variables. This probably new notion appears
implicitly in the statements like embeddings theorems and traces theorems for Sobolev spaces.
In fact, it reveals their nature as theorems about virtual continuity. This notion is especially useful for
the study and classification of measurable functions, as well as in some questions on dynamical systems,
polymorphisms and bistochastic measures. In this work we recall necessary definitions and properties of
admissible metrics, define virtual continuity, describe some of applications.
Detailed analysis is to be presented in another paper.}

\let\thefootnote\relax\footnote{
St. Petersburg   Department   of   V.A.Steklov    Institute   of   Mathematics
RAS, St. Petersburg State University.
E-mail: avershik@gmail.com, paxa239@yandex.ru, fedyapetrov@gmail.com.
The work is supported by RFBR grant 11-01-00677-a, President of Russia grant MK-6133.2013.1, by Chebyshev Laboratory in SPbSU,
Russian Government grant 11.G34.31.0026.}

\section{Introduction. Admissible metrics, Luzin's theorem.}

\subsection{Admissible metrics}

We consider the Lebesgue-Rokhlin standard continuous (atomless) probabilistic measure space, isomorphic to the
unit segment $[0,1]$ with Lebesgue measure.  The first author proposes
\cite{V1,V2,V3} to consider on the fixed standard   space $(X, {\frak A}, \mu)$ different  ({\it admissible}) metrics,
on the contrast to usual approach, when metric space is fixed and different Borel measures are considered.
Such approach is useful and necessary in ergodic theory
and other situations. Agreement of the metric and measure structures leads to the notion of admissible metric triple:

\begin{defe}
A metric or semimetric $\rho$ on the space $X$ is called
\emph{admissible}
if it is measurable, regarded as a function of two variables, on the
Lebesgue space $(X\times X, \mu \times \mu)$ and there exists a subset
$X_0\subset X$ of full measure such that the semimetric space
$(X_0,\rho)$ is separable.

The standard probabilistic space $(X,\mu)$ with admissible (semi-)metric $\rho$ is called {\it admissible metric triple}
or just {\it admissible triple} $(X,\mu,\rho)$.
\end{defe}

Properties of admissible metrics are studied in details by the authors in \cite{ZP}, \cite{VPZ}.
In particular, these works contain several equivalent definitions of admissible triples.

Standartness of the space allows to get the following

\begin{proposition} Let $\rho$ be the admissible metric on $(X, {\frak A}, \mu)$. Then completed Borel sigma-algebra ${\frak B}={\frak B}(X,\rho)$
is a subalgebra of ${\frak A}$. The measure $\mu$ is inner regular with respect to metric $\rho$,
i.e. for any $A \in {\frak A}$ we have
$$ \mu(A)=\sup\{\mu(K)\colon K\subset A,  K \mbox{ is compact in metric } \rho\}.
$$  So, for any admissible metric $\rho$ the measure $\mu$ is Radon measure in the metric space $(X,\rho)$.
\end{proposition}

M. Gromov in the book \cite{G} suggests to consider arbitrary metric triples
$(X,\mu,\rho)$, which he calls $mm$-spaces. Also, Gromov asks the question about their classification,
having in mind classical situations (Riemannian manifolds and so on). It is natural to consider admissible triples
in this framework. Define equivalence of admissible triples up to measure-preserving isometries:
$(X, \mu, \rho)   \sim  (X', \mu', \rho')$,
if
$$\exists T:X\rightarrow X'; \quad T\mu=\mu'; \quad \rho'(Tx,Ty)=\rho (x,y).$$

Here is the main result on this equivalence:
\begin{theorem} {(Gromov \cite{G}; Vershik \cite{V1})}

Consider the map $F_{\rho}: X^{\infty}\times  X^{\infty}\rightarrow M_{\infty}(\Bbb R):$
$$
\quad F_{\rho}(\{x_i,y_j\}_{(i,j) \in \Bbb N\times \Bbb N})=\{\rho(x_i,y_j)\}_{(i,j) \in \Bbb N\times \Bbb N},
$$
and equip infinite product $X^{\infty}\times  X^{\infty}$ by the product-measure $\mu^{\infty}\times \mu^{\infty}$.
Let $D_{\rho}$ denote the measure on the space of matrices (i.e. random matrix of distances), which is
the $F_{\rho}$-image of the measure  $\mu^{\infty}\times \mu^{\infty}$. Call it
MATRIX DISTRIBUTION of the metric $\rho$. It is a complete invariant of above equivalence of admissible metrics.

In other words, $$(X, \mu, \rho)   \sim  (X', \mu', \rho') \Leftrightarrow   D_{\rho}= D_{\rho'}.$$
 \end{theorem}

In \cite{V4} this result is generalized to the so called \emph{pure} measurable functions of several variables.

The following  lemma is useful in the theory of admissible metrics:
\begin{Lm}\label{equivalence}
Let $\rho_1$, $\rho_2$ be admissible semimetrics on the standard space $(X, \mu)$,
and suppose that $\rho_1$ is metric. Then for any $\varepsilon>0$ there exists measurable subset $K \subset X$
such that $\mu(K)>1-\varepsilon$ and semimetric $\rho_2$ (as a function of two variables)
is continuous on $K\times K$ with respect to metric $\rho_1$.
\end{Lm}

We may choose $K$ as a compact subset with respect to admissible metric $\rho=\rho_1+\rho_2$,
if $\mu(K)>1-\eps$.

Lemma immediately implies the

\begin{Co}\label{coincidence}
Let $\rho_1$ and $\rho_2$ be two admissible metrics on the standard space $(X, \mu)$.
Then for any $\varepsilon>0$ there exists $K \subset X$ such that
$\mu(K)>1-\varepsilon$ and topologies defined by metrics $\rho_1$ and $\rho_2$ on $K$ coincide.
\end{Co}

\subsection{Luzin's theorem on measurable functions of one variable}

Furthermore we consider (measurable) real-valued functions, though most of our results remain true
for maps into standard Borel space, in particular into Polish spaces.
Egorov's and Luzin's classical theorems on measurable functions of one variable are well-known.
The generalized Luzin's theorem for arbitrary admissible triple follows from above results:

\begin{Co}[Luzin's theorem] Let  $\rho$ be an admissible metric on the standard space $(X,\mu)$,
let $f$ be a measurable map from $X$ into Polish space $(M,d)$. Then for any $\varepsilon>0$ there exists
a measurable subset $K \subset X$ such that $\mu(K)>1-\varepsilon$ and $f$ is continuous on $K$ with respect to
metric $\rho$.
\end{Co}

\begin{proof} Set $\rho_1(x,y)=\rho(x,y)+d(f(x),f(y))$. Then $\rho_1$ is a trivial example of
an admissible metric, with respect to which $f$ is continuous. By
\ref{coincidence} there exist a subset $K$ having measure $\mu(K)>1-\eps$, on which this
continuity implies continuity with respect to $\rho$.
\end{proof}

But this fact does not hold true for functions of several variables.

\section{Virtual continuity}

\subsection{Definitions and first examples}
Let $f(\cdot,\cdot)$ be a measurable function of two variables. Then Luzin's theorem analogue
(continuity on the product $X'\times Y'$ of sets of measure $>1-\eps$ with respect to given metric
$\rho[(x_1,y_1),(x_2,y_2)]=\rho_X(x_1,x_2)+\rho_Y(y_1,y_2)$) is not in general true.
This leads to the following key notion of this work.
(Sum of metrics may be replaced to maximum or other metric defining the topology of direct product.
To stress this we denote generic metric with such topology by $\rho_X \times \rho_Y$).

\begin{defe} Measurable function $f(\cdot,\cdot)$ on the product $(X,\mu)\times (Y,\nu)$ of standard spaces
is called \emph{virtually continuous}, if for any $\varepsilon  >0$ there exist sets $X'\subset X,  Y'\subset Y$,
each of which having measure $1-\eps$, and admissible semimetrics $\rho_X,\rho_Y$ on $X',Y'$
respectively such that function  $f$ is continuous on  $(X'\times Y',\rho_X \times \rho_Y)$.
virtual functions of several variables are defined in the same way.
\end{defe}
It is essential that admissible metric with respect to which function becomes continuous is not arbitrary,
but respects the structure of direct product (in more general setting, it respects selected subalgebras, see further).
It is easy to verify that there does not exist universal metric of such type (i.e. such a metric that virtual
continuity implies continuity in this metric). It explains the non-trivial properties of defined notion.

It is clear that any admissible metric (considered as a function of two variables) is virtually continuous.
So is any function, which is continuous with the respect to product of admissible metrics.
Degenerated functions (or ``finite rank functions'') $f(x,y)=\sum_{i=1}^n \varphi_i(x)\psi_i(y)$, where $\phi_i(\cdot), \psi_i(\cdot)$, $i=1,\dots n$
are arbitrary measurable functions, are also virtually continuous. For the proof just use Luzin's theorem for all functions
$\varphi_i(\cdot)$,  $i=1 \dots n$, and $\psi_i(\cdot)$,  $i=1 \dots n$.

less trivial examples of virtually continuous functions are given by functions from some Sobolev spaces
and kernels of trace class operators.
For virtually continuous functions there exist well-define restrictions on some subsets of zero measure --- concretely,
onto supporters of (quasi)bistochastic measures, see next paragraph.

An easy example of not virtually continuous measurable function on $[0,1]^2$ is provided by the
characteristic function of the triangle $\{x\geq y\}$. In general, for functions on the square of a compact group
depending of the ratio of variables the criterion of virtual continuity is simple:

\begin{proposition}\label{gruppa} Let $G$ be a metrizable compact group, $f$ be a Haar measurable function on
$G$. Then the function $F(x,y):=f(xy^{-1})$ on $G\times G$ is virtually continuous if and only if
$f$ is equivalent to a continuous function.
\end{proposition}

Stress once more that the definition of virtual
continuity is not topological, but measure-theoretical in nature.
It applies to the choice of various metrics on the measure space.
So, the direct sense of the proposition \ref{gruppa} is that
group structure and measure-theoretical structure allow to
reconstruct topology.

\subsection{Bistochastic measures and polymorphisms}

From the measure-theoretical point of view a function of $k$ variables
on the product of standard continuous spaces is nothing but the function on
the standard continuous space (due to isomorphism of all such spaces).
In order to deal with it as a function of $k$ variables, we have to introduce another
category, then just measurable spaces.

namely, consider the following structure: the measure space $({\cal X},{\frak A}, m)$,
with $k$ selected sigma-subalgebras ${\frak A_1},\dots,{\frak A_k}$ in $\frak A$.
It is natural to suppose that those subalgebras generate the whole sigma-algebra $\frak A$.

The connection with general viewpoint is the following: in the space
${\cal X}=\prod_{i=1}^k (X_i,{\frak A_i},\mu_i)$, $m=\prod \mu_i$,
identify algebras ${\frak A_i}$ with subalgebras of
${\frak A}=\prod {\frak A}_i$ by multiplying to trivial
subalgebras on other multiples. In other words, function $f(x_1,\dots,x_k)$ on ${\cal X}$
is ${\frak A}_i$-measurable iff $f$ depends only on $i$-th variable $x_i$ ($i=1, \dots, k$).

 \begin{defe}
Measurable function on the standard space ${\cal X}$ with $k$ selected subalgebras, which generate
the whole sigma-algebra, is called \it{general measurable function of $k$ variables.}
\end{defe}

  Consider some measure $\lambda$ on the sigma-algebra $\frak A$.
 It may be restricted onto sigma-subalgebras ${\frak A}_i$, $i=1,\dots,k$.
Let's consider such measures $\lambda$ for which those restrictions are
absolutely continuous with respect to the measure $m$ (restricted onto ${\frak A}_i$).
  If restrictions of  $\lambda$ onto ${\frak A}_i$ coincide with $m$,
	$i=1 \dots k$, such a measure $\lambda$ is called  {\it multistochastic} with respect to given subalgebras
	(bistochastic for  $k=2$); if restrictions are just equivalent to $m$ for $i=1, \dots,k$,
	we call $\lambda$  {\it almost multistochastic}. Finally, if
  $\lambda(U)\leq m(U)$ for any $U\in {\frak A}_i,i=1,\dots,k$, we call $\lambda$
	{\it submultistochastic}.

Of course, bistochastic measure on
$X\times Y$ may be singular with respect to the product measure. For instance, in the case of direct product
of segments $(X,\mu)=(Y,\nu)=[0,1]$ there is a bistochastic measure $\lambda$ on diagonal $\{x=y\}$ (with density $d\mu(x)$).

 Furthermore we suppose for simplicity that $k=2$, i.e. consider functions of two variables.
But there is no serious difference for $k>2$. We consider not only independent variables, most of the notions
may be defined for general pair of sigma-algebras. But even the case of independent variables is often useful
to treat as a general case.

Bistochastic measure on the direct product of spaces define the so called
 \emph{polymorphism} of the space $(X,\mu)$ into  $(Y,\nu)$ (see \cite{V5}),
i.e. ``multivalued mapping'' with invariant measure.
  The case of identified variables $(X,{\frak A},\mu) = (Y,{\frak B},\nu)$ is of special interest: polymorphism
	in this case 	generalizes the concept of automorphism of measure space.
	Almost bistochastic measures defines a  {\it polymorphism with quasi invariant measure}.
	Bistochastic or almost bistochastic measure $\lambda$ defines also a bilinear
  (in general case $k$-linear) form $(f(x),g(y))\rightarrow \int f(x)g(y)d\lambda(x,y)$,
	corresponding to the so called Markovian, resp. quasi Markovian operator in corresponding functional
	spaces. Note that this operator $U_{\lambda}$ is a contraction, i.e. has norm at most 1, which preserves the cone of non-negative
	functions. In the case of bistochastic measure this operator (as well as adjoint operator) preserve constants: $U_{\lambda}1=1$.
	
  See \cite{V5,V6,V7} about many connections of polymorphisms  (Markovian operators, joinings, couplings, correspondences, Young measures, bibundles etc).
	Bistochastic measures play a key role in the intensively developing theory of continuous graphs \cite{L}.

\subsection{Further properties of virtually continuous functions}

First of all, virtually continuous functions enjoy
the properties a priori stronger than required in the definition.
On the one side, we may require for sets $X',Y'$ from the definition to have full measure:

\begin{theorem}\label{polnayamera} Let function $f(\cdot,\cdot)$ be virtually continuous.
Then there exist sets $X'\subset X$, $Y'\subset Y$ of full measure and admissible semimetrics
$\rho_X,\rho_Y$ on $X',Y'$ respectively such that
 $f$ is continuous on $(X'\times Y',\rho_X\times \rho_Y)$.
 \end{theorem}

 On the other hand, we may fix arbitrary admissible metrics:

\begin{theorem}\label{arbitrarymetrics}
Let function $f(\cdot,\cdot)$ be virtually continuous.
Then for any admissible semimetrics $\rho_X,\rho_Y$ on $X',Y'$ and for any $\varepsilon  >0$ there exist sets $X'\subset X,  Y'\subset Y$,
each of which having measure $1-\eps$, such that function  $f$ is continuous on  $(X'\times Y',\rho_X \times \rho_Y)$.
 \end{theorem}

A function of two variables on $X\times Y$ may be treated as a map from $X$ into the space of functions on $Y$
(i.e. $f(x,y)\equiv f_x(y)$). In \cite{VS} such a viewpoint is used for classification problem.
Virtual continuity is described in those terms by the following equivalent definition:

\begin{theorem}\label{compactness} Virtual continuity of the function $f(\cdot,\cdot)$
is equivalent to the following property of a function: for any $\eps>0$ there exist
sets $X'\subset X$,  $Y'\subset Y$ having measures not less then $1-\eps$,  such that the set of functions
$f_x(\cdot)$ on $Y'$ (variable $x$ runs over $X'$) form a totally bounded (precompact) family in $L_{\infty}(Y')$.
\end{theorem}

This theorem-definition has an important corollary:
continuity in one variable implies virtual continuity.

\begin{Lm}\label{razd}
Let $(X,\mu)$, $(Y,\nu)$
be standard continuous probabilistic spaces, $\rho_Y$ be an admissible metric on the set $Y'$ of full
measure. Let measurable function $f:X\times Y \to \mathbb{R}$ be so that functions $f(x, \cdot)$ are continuous on $(Y', \rho_Y)$
for $\mu$-almost all $x\in X$. Then $f$ is virtually continuous.
\end{Lm}

Theorem \ref{polnayamera} immediately implies the converse: for any virtually continuous function $f$
such a metric $\rho_Y$ exists. So, the statement of Lemma \ref{razd}
(``continuity by appropriate metric on $y$ for almost all fixed $x$'') is equivalent to virtual continuity.

It's remarkable that the spaces $X$ and $Y$ (i.e. arguments of the function) play different roles in this definition.
However, a posteriori the property appears to be symmetric under the change of order of variables.
This is another demonstration of the non-triviality of the virtual continuity concept.

As we have seen, measurable functions $f(\cdot,\cdot)$ are classified by matrix distributions, i.e. by measures
on the space of matrices $(a_{ij})_{i,j=1}^\infty$, induced by the map $f\rightarrow (a_{i,j}=f(x_i,y_j))$ (points $x_i$ in $X$
and $y_i$ in $Y$, $i=1,2,\dots$ are chosen independently). Virtual continuity also may be characterized on this manner:

\begin{theorem} Let $x_1, x_2, \dots$ (resp. $y_1, y_2, \dots$) be independent random points in $X$ (resp. in $Y$).
Virtual continuity of the measurable function $f(x,y)$ is equivalent to the following condition: for any $\varepsilon >0$
there exist positive integer $N$ such that the following event has probability 1:

 there exist such a partitions of naturals
 $\{1, 2, \dots,\}=\sqcup_{i=0}^N A_i=\sqcup_{i=0}^N B_i$ that
 upper density of the set $A_0\cup B_0$ is less than $\eps$ (i.e. $\limsup |(A_0\cup B_0)\cap [1,n]|/n<\eps$) and
 $|f(x_s,y_t)-f(x_r,y_p)|<\varepsilon $  for all $i,j>0$, $s,r\in A_i$, $p,t\in B_j$.
\end{theorem}

Aforementioned characteristics of the virtual continuity allow to deduce that virtual continuous functions form
a nowhere dense set in the space of all measurable functions (with measure convergence topology).

 \subsection{Thickness}
Consider the space $X\times Y$ with product measure $\mu \times \nu$.
Choose two subalgebras in its sigma-algebra, defined by projections onto $X$ and $Y$.
For measurable set $Z\subset X\times Y$ define its \emph{thickness} $\thi(Z)$ as the infimum of the value
$\mu(X_1)+\nu(Y_1)$ taken by all pairs of measurable sets $X_1\subset X,Y_1\subset Y$ such that
\begin{equation}\label{thickness}
\mu\times \nu (Z\setminus(X_1\times Y\cup X\times Y_1))=0.
\end{equation}

Sets of the form $X_1\times Y$, $X\times Y_1$ are exactly sets from our sigma-subalgebras. It allows
to define generalized thickness for other selected subalgebras in the standard space.

The following properties of thickness are immediate:
\begin{itemize}
\item thickness of a set does not exceed 1 and equals 0 for and only for sets of measure 0;
\item thickness of a subset does not exceed a thickness of a set;
\item thickness of a set does not exceed its measure;
\item thickness of a finite or countable union of sets does not exceed sum of thicknesses.
\end{itemize}

The following lemma is slightly less obvious.

\begin{lemma}\label{thick}
If $\thi(Z)=0$, then $X_1,Y_1$ of zero measure may be chosen in the definition \ref{thickness}.
\end{lemma}

By ``arbitrarily thin set'' we mean ``the set of arbitrarily small thickness''.
It allows to reformulate lemma \ref{thick} as follows:
if the set may covered may arbitrarily thin set, then its complement contains  mod 0 the product
of full measure sets.

the following equivalent definition of thickness is in some situations more appropriate for
using:
\begin{lemma}\label{thick2}
For any set $Z$ consider pairs of measurable functions $f:X \to [0,1]$, $g:Y\to [0,1]$,
such that $f(x)+g(y)\geq \chi_Z(x,y)$ for $\mu\times \nu$-almost all pairs $(x,y)$.
Then thickness of $Z$ is infimum of the sum of integrals $\int_X f d\mu$ and $\int_Y g d\nu$.
\end{lemma}

Applying this lemma and choosing weakly convergent subsequence the lower semicontinuity of
thickness may be proved:
\begin{lemma}
Let $\{Z_n\}$ -be increasing sequence of measurable sets, $Z=\cup_n Z_n$. Then $\thi(Z)=\lim\thi(Z_n)$.
\end{lemma}

Note that there is no upper semicontinuity: all sets $\{(x,y):0<|x-y|<1/n\}\subset [0,1]^2$ have thickness 1, while
their intersection is empty.

Define the convergence of functions ``in thickness'' analogously to convergence in measure. This is convergence
in the metrizable topology defined by the following distance:

\begin{defe}
Define the distance $\tau(f(x,y),g(x,y))$ between arbitrary measurable functions of two variables as infimum of
such $\eps>0$ that $\thi\{(x,y):|f(x,y)-g(x,y)|>\eps\}\leq \eps$.
 \end{defe}

Convergence in this $\tau$-metric implies convergence in measure (but not vice versa).

Let $\xi_X: X=\sqcup_{i=1}^n X_i$, $\xi_Y: Y=\sqcup_{i=1}^m Y_i$ be finite partitions of the spaces $X,Y$ respectively
onto measurable subsets of positive measure.
Functions which are constant mod $0$ on each product $X_i\times Y_j$ are called \emph{step functions}.
Finite linear combinations $\sum_{i=1}^N a_i(x)b_i(y)$ are called \emph{functions of finite rank}.
The set of measurable functions is complete in $\tau$-metric.

The following theorem connects finite rank functions and virtual continuity.

\begin{theorem}\label{tau-virtual}
The $\tau$-closure of the set of step functions (or the set of finite rank functions)
is exactly the set of virtually continuous functions.
\end{theorem}

This definition of virtual continuity is even more explicitly measure-theoretical,
it does not appeal to metrics at all. Also, it has clear generalisation to arbitrary pair
of sigma-subalgebras. See \cite{S} on close concepts.

\subsection{Norm in the space of virtually continuous functions}

Defined convergence in $\tau$-metric is a virtual continuity analogue of the convergence in measure.
Known Banach spaces of measurable functions also have their analogues.

A measurable function $h(\cdot,\cdot)$ on the space $(X\times Y,\mu\times \nu)$ is called subbistochastic, if the measure with $\mu\times \nu$-density
 $|h(\cdot,\cdot)|$ is subbistochastic. Denote by  $\mathcal{S}$ the set of subbistochastic functions.

Define a finite or infinite norm of a measurable function $f(\cdot,\cdot)$ as

$$
\|f\|:=\inf\{\int_X |a(x)|d\mu(x)+\int_Y |b(y)|d\nu(y): |f(x,y)|\leq a(x)+b(y) \text{ a.e.}\}
$$

Next theorem is an analogue of known L. V. Kantorovich's duality theorem \cite{K} in the mass transportation problem
(concretely, of duality between measures space with Kantorovich distance and and the space of Lipschitz functions, see
also \cite{VK}).
\begin{theorem}\label{dual}
$$
\|f\|=
\sup \{\int_{X\times Y} |f(x,y)| h(x,y) dx dy: h\in \mathcal{S}\}.
$$
\end{theorem}
Coincidence of infimum and supremum is a duality statement in infinite-dimensional linear programming.
But in our case the proof requires more delicate arguments than the Monge-Kantorovich problem.
The reason is that we consider $L^1$-type space, in which the cone of non-negative functions has empty interior
(on the contrast to the space of continuous functions, used in the transport problem).
This does not allow to apply directly standard separability theorems.

\begin{theorem}
The closure of step functions in above norm consists exactly of all virtually continuous functions
having finite norm (in particular, each bounded virtually continuous function belongs
to this closure).
\end{theorem}

Denote this space by $VC^1$. It is an analogue of the space $L^1$ for virtually continuous functions
and is a pre-dual for the space of polymorphisms with bounded densities of projections.

\begin{theorem}\label{dualspace}
The space dual to $VC^1$ is a space of signed measures $\eta$ on $X \times Y$ with finite norm
$$\|\eta\|_{\rm{me}}=\max\{\Big|\Big|\frac{\partial P^x_*|\eta|}{\partial \mu}\Big|\Big|_{L^\infty(X,\mu)},\Big|\Big|\frac{\partial P^y_*|\eta|}{\partial \nu}\Big|\Big|_{L^\infty(Y,\nu)}\},$$
where $P^x$ and $P^y$ are projections onto $X$ and $Y$ respectively and $|\eta|$ is a full variation of a signed measure $\eta$.
\end{theorem}

\begin{corollary} Virtually continuous functions from $VC^1$
(in particular, bounded virtually continuous functions)
have a well defined integral not only over sets of positive measure (as all summable functions do have),
but also also over bistochastic (singular) measures: for instance, over Lebesgue measure on
diagonal $\{x=y\}\subset [0,1]^2$, or over measure concentrated on a graph
of a map with quasiinvariant measure.
To summarize, virtually continuous functions have traces (restrictions) on diagonal
and other subsets in the sense of Sobolev trace theorems.
\end{corollary}

 \section{Applications: embeddings theorems, trace theorems and restrictions of metrics}\label{tvlo}

Here we mention some applications of virtual continuity.

\subsection{Sobolev spaces and trace theorems}

\begin{theorem}\label{sob} Let $\Omega_1,\Omega_2$ be domains of dimensions $d_1,d_2$ respectively, suppose that
$pl>d_2$ or $p=1,l=d_2$. Then functions from the Sobolev space $W_p^l(\Omega_1\times \Omega_2)$
($l$-th generalized derivatives are summable with power $p$) are virtually continuous as functions of two
variables $x\in \Omega_1$, $y\in \Omega_2$. Embedding $W_p^l(\Omega_1\times \Omega_2)$ into $VC^1 (\Omega_1,K)$
is continuous for any compact subset $K$ of the domain $\Omega_2$.
\end{theorem}

\begin{proof}
Using the theorem of embedding of Sobolev space into continuous functions (see, for instance, \cite{Ma,AF}),
we have the following estimate for functions $h(y)\in W_p^l(\Omega_2)$:
$$
\|h\|_{C(K)}\leq c(\Omega_2,K) \|h\|_{W_p^l(\Omega_2)}.
$$
Let $f(x,y)\in W_p^l(\Omega_1\times \Omega_2)$ be a smooth function.
Set $$a(x):=\|f(x,\cdot)\|_{W_p^l(\Omega_2)}.$$
Then by Fubini's theorem $a\in L^1(\Omega_1)$ and
$$
\int |a|\leq c(\Omega_1,\Omega_2) \|f(x,y)\|_{W_p^l(\Omega_1\times \Omega_2)}.
$$
The following estimate holds on $\Omega_1\times K$:
$$
|f(x,y)|\leq \|f(x,\cdot)\|_{C(K)}\leq c(\Omega_2,K) a(x).
$$
Summarizing this we have
\begin{equation}\label{vlozhenie}
\|f\|_{VC^1(\Omega_1,K)}\leq c(\Omega_1,\Omega_2,K) \|f\|_{W_p^l(\Omega_1\times \Omega_2)}.
\end{equation}
Each function in the class
 $W_p^l(\Omega_1\times \Omega_2)$ is a limit of a sequence of smooth functions,
by (\ref{vlozhenie}) it is a limit in $VC^1$ as well.
\end{proof}

So, under conditions of this theorem we may integrate functions over
quasibistochastic measures. It generalizes usual theorems about traces on submanifolds.

\subsection{Nuclear operators in Hilbert space}

It is well known that the space of nuclear operators in the Hilbert space
$L^2$ is a projective tensor product of Hilbert spaces.
Their kernels are measurable functions of two variables, which can hardly be described directly.
the following theorem claims that kernels of nuclear operators are virtually
continuous as functions of two variables.
Note that kernels of Hilbert-Schmidt operators are not in general
virtually continuous.

\begin{theorem}
Let $(X,\mu)$, $(Y,\nu)$ be standard spaces. the space of kernels of nuclear operators from $L^2(X)$ to $L^2(Y)$
(with Schatten -- von Neumann norm) embeds continuously into $VC^1$.
\end{theorem}
It implies that such kernels may be integrated not only over diagonal when $X=Y$, which is well known,
but by bistochastic measures. But the space $VC^1$ is wider than kernels of nuclear
operators. If we look at  $VC^1$ as to the space of kernels of integral operators,
it is not unitray invariant, on the contrast to Schatten--von Neumann spaces.
indeed, the definition of $VC^1$ essentially uses known
sigma-subalgebras, which do not have necessary invariance. Close question is
considered in \cite{GD}.

\subsection{Restrictions of metrics}

The following problem was one of origins of this paper.
Let $(X,\mu)$ be a standard space with continuous measure.
Assume that $\rho$ is an admissible metric and $\xi$ is a measurable partition of $(X,\mu)$ with parts of null measure
(say, $\xi$ is a partition onto level sets of function which is not constant on sets of positive measure).
May we correctly restrict our metric (a s a function of two variables) onto elements of this partition?

It is not immediately clear, since the metric is a priori just a measurable function.
But admissible metric is virtually continuous, and so for our goal it suffices to
define a bistochastic measure, onto which we have to restrict it. Suppose for simplicity that
$X=[0,1]^2$, $\mu$ is a Lebesgue measure, $\xi$ is a partition onto vertical lines.
Then we say about restriction of virtually continuous function defined on $X^2=[0,1]^4$ onto
three-dimensional submanifold
 $\{(x_1,x_2,x_3,x_4):x_1=x_3\}$.
It is easy to see that such a submanifold equipped by a three-dimensional Lebesgue  measure
defines a bistochastic measure on $X\times X$.

\section{Acknowledgements}

We are grateful to L.Lovasz, who  sent us his recent monograph \cite{L}, in which close questions
are discussed, and to A.Logunov for paying our attention to the
possibility of dual definition of the thickness.

\end{document}